\documentclass[10pt,twoside]{siamltex}
\usepackage{amsfonts,epsfig}
\usepackage{amssymb}
\usepackage{eucal}
\usepackage{amsmath}

\usepackage[active]{srcltx} 

\newtheorem{remark}[theorem]{Remark}

\newcommand{\complex}{\mathbb{C}}

\newcommand{\nc}{\newcommand}

\nc{\twovec}[2]{\left(\begin{array}{c} #1 \\ #2\end{array} \right )}
\nc{\threevec}[3]{\!\left(\!\!\begin{array}{c} #1 \\ #2 \\ #3 \end{array}\!\!\right )\!}
\nc{\twomatrix}[4]{\left(\begin{array}{cc} #1 & #2\\ #3 & #4 \end{array} \right)}
\nc{\threematrix}[9]{{\!\left(\!\!\begin{array}{ccc} #1 & #2 & #3\\ #4 & #5 & #6 \\ #7 & #8 & #9 \end{array} \!\!\right)\!}}
\nc{\twodet}[4]{\!\left|\!\!\begin{array}{cc} #1\! & \!#2\\ #3\! & \!#4 \end{array} \!\!\right|}

\begin{document}



\bibliographystyle{plain}
\title{
PostLie algebra structures on the Lie algebra $\mathrm{sl}(2,\complex)$}

\author{
Yu Pan\thanks{School of Mathematics, Nankai University, Tianjin 300071, China (pauline65@126.com,
\newline
liuqing1989930@163.com).}
\and
Qing Liu\footnotemark[1]
\and
Chengming Bai \thanks{Chern Institute of Mathematics \& LPMC, Nankai University, Tianjin 300071, China (baicm@nankai.edu.cn)}
\and
Li Guo \thanks{Department of Mathematics and Computer Science, Rutgers University, Newark, NJ 07102, USA (liguo@rutgers.edu)}
}


\pagestyle{myheadings}
\markboth{Y.\ Pan, Q.\ Liu, C.\ Bai, and L. Guo}{PostLie Algebra Structures on the Lie Algebra $\mathrm{sl}(2,\complex)$}
\maketitle

\begin{abstract}
The PostLie algebra is an enriched structure of the Lie algebra that has recently arisen from operadic study. It is closely related to pre-Lie algebra, Rota-Baxter algebra, dendriform trialgebra, modified classical Yang-Baxter equations and integrable systems. We give a complete classification of PostLie algebra structures on the Lie algebra $\mathrm{sl}(2,\complex)$ up to isomorphism.
We first reduce the classification problem to solving an equation of $3\times 3$ matrices. To solve the latter problem, we make use of the classification of complex symmetric matrices up to the congruent action of orthogonal groups.
\end{abstract}



\section{Introduction}

We begin with recalling background on PostLie algebras~\cite{Va1}.

\begin{definition}
 {\rm
\begin{enumerate}
\item
 A {\bf (left) PostLie $\complex$-algebra} is a
$\complex$-vector space $L$ with two binary operations $\circ$
and $[,]$ which satisfy the relations:
\begin{equation}
[x,y]=-[y,x],\label{eq:polie1}
\end{equation}
\begin{equation}
[[x,y],z]+[[z,x],y]+[[y,z],x]=0,\label{eq:polie2}
\end{equation}
\begin{equation}
z\circ(y\circ x)-y\circ(z\circ x)+(y\circ z)\circ x-(z\circ y)\circ
x+[y,z]\circ x=0,\label{eq:polie3}
\end{equation}
\begin{equation}
z\circ[x,y]-[z\circ x,y]-[x,z\circ y]=0,\label{eq:polie4}
\end{equation}
for all $x,y\in L$.
\item
Let $(\mathfrak{G}(L),[,])$ denote the Lie algebra defined by Eq.~(\ref{eq:polie1}) and Eq.~(\ref{eq:polie2}). Call $(L,[,],\circ)$ a {\bf PostLie algebra on
$(\mathfrak{G}(L),[,])$}.
\item
Let $(\mathfrak{g},[,])$ be a Lie algebra. Two PostLie algebras $(\mathfrak{g},[,],\circ)$ and $\mathfrak{g},[,],\star)$ on the Lie algebra $\mathfrak{g}$ are called {\bf isomorphic on the Lie algebra $(\mathfrak{g},[,])$} if there is an automorphism $f$ of the Lie algebra $(\mathfrak{g},[,])$ such that
$$ f(x\circ y)=f(x)\star f(y), \quad \forall x,y\in \mathfrak{g}.$$
\end{enumerate}
}
\label{de:postlie}
\end{definition}

The concept of a PostLie algebra was recently introduced by Vallette from an operadic study~\cite{Va1}. It is closely related to pre-Lie algebra, Rota-Baxter algebra, dendriform trialgebra and modified classical Yang-Baxter equation, and has found applications to integrable systems~\cite{BGN}.
Further the corresponding operad plays the role of ``splitting" a binary quadratic operad into three pieces in terms of Manin black products~\cite{Va2}.

It is important to give examples before attempting to achieve certain classification in an algebraic structure. Considering the great challenge in giving a complete
classification for the well-known algebra structures such as Lie algebras and associative algebras, it is reasonable
to begin with studying the classification of PostLie algebras on some well-behaved Lie algebras, such as complex semisimple Lie algebras. Thus as a first step and as a guide for further investigations, we determine all isomorphic classes of PostLie algebra structures on the Lie algebra $(\mathrm{sl}(2,\complex),[,])$.

Let
\begin{equation}
e_1:=\frac{1}{2}\twomatrix{0}{1}{-1}{0}, \quad e_2:=\frac{1}{2\sqrt{-1}}\twomatrix{0}{1}{1}{0}, \quad e_3:=\frac{1}{2\sqrt{-1}}\twomatrix{1}{0}{0}{-1}.
\label{eq:base}
\end{equation}
They form a $\complex$-linear basis of $\mathrm{sl}(2,\complex)$ and determine the Lie algebra $(\mathrm{sl}(2,\complex),[,])$ through the relations
\begin{equation}
[e_{2},e_{3}]=e_{1},\quad [e_{3},e_{1}]=e_{2}, \quad [e_{1},e_{2}]=e_{3}.
\label{eq:eprod}
\end{equation}

Our main result of this paper is the following classification theorem.
\begin{theorem}
The following is a complete set of representatives for the isomorphic classes of PostLie algebras $(\mathrm{sl}(2,\complex), [,], \circ)$
 on the Lie algebra $(\mathrm{sl}(2,\complex),[,])$.
\begin{enumerate}
\item
$e_i\circ e_j=0, i,j=1,2,3$;
\item
$e_i\circ e_j = [-e_i,e_j], i,j=1,2,3$;
\item
$e_{1}\circ e_{i}=[-e_{1},e_{i}],
\quad e_{2}\circ e_{i}=[-\frac{1+\sqrt{-1}}{2}e_{2}+\frac{\sqrt{-1}-1}{2}e_{3},e_{i}],
\\
e_{3}\circ e_{i}=[-\frac{1+\sqrt{-1}}{2}e_{2}+\frac{\sqrt{-1}-1}{2}e_{3},e_{i}], \quad i=1,2,3;$
\item
$e_{1}\circ e_{i}=[(\sqrt{-1}-\frac{1}{2})e_{1}+(1-\frac{\sqrt{-1}}{2})e_{2},e_{i}],
\\
e_{2}\circ e_{i}=[(1+\frac{\sqrt{-1}}{2})e_{1}-(\sqrt{-1}+\frac{1}{2})e_{2},e_{i}],
\quad e_{3}\circ e_{i}=0,\quad i=1,2,3;$
\item
$e_{1}\circ e_{i}=[k e_{1},e_{i}],
\quad e_{2}\circ e_{i}=[-\frac{1}{2}e_{2}+\frac{\sqrt{-1}}{2}e_{3},e_{i}],
\quad e_{3}\circ e_{i}=[-\frac{\sqrt{-1}}{2}e_{2}-\frac{1}{2}e_{3},e_{i}],\\ i=1,2,3 ,\quad k\in \complex.$
\end{enumerate}
\label{thm:class}
\end{theorem}

As pointed out in~\cite{BGN}, such a classification problem is related to the classification of the modified classical Yang-Baxter equation
introduced by Semenov-Tian-Shansky in~\cite{S}. While the classification in~\cite{S} works for the so-called graded r-matrices in a finite-dimensional semisimple Lie algebra, in terms of extensions of linear operators associated to the parabolic subalgebras, our classification in the case that we are considering is given without any constraint graded conditions and is more precise in the sense that the structural constants are spelled out explicitly.

Our proof of the theorem consists of two steps:
\begin{enumerate}
\item[{\bf Step 1.}]
Give a one-one correspondence from the isomorphic classes of PostLie algebra structures on $(\mathrm{sl}(2,\complex),[,])$ to the congruent classes of solutions of the matrix equation Eq.~(\ref{eq:mateq}). This will be carried out in Section~\ref{sec:tomatrix}.
\item[{\bf Step 2.}]
Classify the congruent classes of solutions of the matrix equation Eq.~(\ref{eq:mateq}). This will be carried out in Section~\ref{sec:matclass}.

For this purpose, we make use of a result on the canonical forms for complex symmetric matrices under the congruent action of $\mathrm{SO}(3,\complex)$ (Proposition~\ref{pp:clale}).
When a solution $A$ of Eq.~(\ref{eq:mateq}) has full rank, it can be shown that $A$ is symmetric. Thus we only need to check against Eq.~(\ref{eq:mateq}) the complex symmetric matrices with full rank which are in the above canonical forms. When a solution $A$ does not have full rank, $A$ is no longer symmetric. Then we try to relate Eq.~(\ref{eq:mateq}) of $A$ to equations of various symmetrizations of $A$, such as $A'A$ and $A'+A$ so that we can still apply Proposition~\ref{pp:clale}. This strategy turns out to work quite nicely.
\end{enumerate}

\section{A matrix equation from PostLie algebras}
\label{sec:tomatrix}

In this section, we carry out the first step in establishing Theorem~\ref{thm:class} by proving Theorem~\ref{thm:mat1}. We begin with recalling two results on PostLie algebras.

\begin{lemma} {\rm (\cite{Va1})} Let $(L,[, ],\circ)$ be a PostLie algebra. Then the binary operation given by
\begin{equation}
\{ x,y\}:=x\circ y-y\circ x+[x,y],\;\;\forall x,y\in L,
\end{equation}
defines a Lie algebra.
\label{lem:va}
\end{lemma}

\begin{proposition} {\rm (\cite{BGN})} Let $(\mathfrak g,[,])$ be a semisimple Lie algebra. Then any PostLie algebra structure $(\mathfrak g, [, ], \circ)$
(on $(\mathfrak g,[,])$) is given by
\begin{equation}
x\circ y=[f(x),y],\;\;\forall x,y\in \mathfrak g,\end{equation}
where $f:\mathfrak g\rightarrow \mathfrak g$ is a linear map satisfying
\begin{equation}
[f(x),f(y)]=f([f(x),y]+[x,f(y)]+[x,y]),\;\;\forall x,y\in \mathfrak g.\label{eq:RB1}
\end{equation}\label{prop:map}
\end{proposition}
\begin{remark}
{\rm
\begin{enumerate}
\item
A linear map $f$ satisfying Eq.~(\ref{eq:RB1}) is called a {\bf Rota-Baxter operator of weight $1$}~\cite{G}.
\label{it:rb1}
\item
We also note that Eq.~(\ref{eq:RB1}) is equivalent to the condition that $[f(x),f(y)]=f(\{x,y\})$, for any $x,y\in \mathfrak g$,
that is, $f$ is a homomorphism between the two Lie algebras $(\mathfrak g, [,\ ])$ and $(\mathfrak g, \{, \})$ from Lemma~\ref{lem:va}.
\label{it:rb2}
\end{enumerate}
}
\label{rk:rb}
\end{remark}

\begin{theorem}
Let $\circ$ be a binary operation on $\mathrm{sl}(2,\complex)$. The following statements are equivalent.
\begin{enumerate}
\item
The triple $(\mathrm{sl}(2,\complex),[,],\circ)$ is a PostLie algebra on $(\mathrm{sl}(2,\complex),[,])$;
\label{it:postlie}
\item
The operation $\circ$ is given by
\begin{equation}
x\circ y=[f(x),y], \quad \forall x,y\in \mathrm{sl}(2,\complex),
\label{eq:circf}
\end{equation}
where $f: \mathrm{sl}(2,\complex)\to \mathrm{sl}(2,\complex)$ is a linear map satisfying Eq.~(\ref{eq:RB1});
\label{it:map}
\item
The operation $\circ$ is given by
$$x\circ y=[f(x),y], \quad \forall x,y\in \mathrm{sl}(2,\complex),$$
where $f: \mathrm{sl}(2,\complex)\to \mathrm{sl}(2,\complex)$ is a linear map whose matrix $A$ with respect to the basis $\{e_1,e_2,e_3\}$ satisfies
\begin{equation}
A'((\mathrm{tr}(A)+1)I_3-A)=A^*.
\label{eq:mateq}
\end{equation}\
Here $A'$ is the transpose matrix of $A$ and $A^*$ is the adjugate matrix of $A$.
\label{it:mat}
\end{enumerate}
Furthermore, the linear map $f$ (and hence its matrix $A$) in Item~(\ref{it:mat}) is unique for a given $\circ$.
\label{thm:mat1}
\end{theorem}
Because of their uniqueness, the linear map $f$ (resp. its matrix $A$) in the theorem is called the {\bf linear map} (resp. the {\bf matrix) of the PostLie algebra $(\mathrm{sl}(2,\complex),[,],\circ)$} and is denoted $f_\circ$ (resp. $A_\circ$).

\begin{proof}
(\ref{it:postlie}) $\Longrightarrow$ (\ref{it:map}) is a special case of Proposition~\ref{prop:map}.
\smallskip

\noindent
(\ref{it:map}) $\Longrightarrow$ (\ref{it:postlie}): Applying Eq.~(\ref{eq:RB1}) to $z\in \mathrm{sl}(2,\complex)$ gives Eq.~(\ref{eq:polie3}). Since the left multiplication $y\mapsto [f(x),y], y\in \mathrm{sl}(2,\complex)$ is a derivation, from Eq.~(\ref{eq:circf}) we obtain Eq.~(\ref{eq:polie4}). \smallskip

\noindent
(\ref{it:map}) $\Longrightarrow$ (\ref{it:mat}):
Set
$$f(e_i)=\sum_{j=1}^3a_{ij}e_j,\;\;{\rm where}\;\;A=(a_{ij}),\;\;
  a_{ij}\in \complex,i,j=1,2,3.$$
Substituting the above equation into Eq.~(\ref{eq:RB1}), we obtain
{\allowdisplaybreaks
\begin{eqnarray*}        &&f([a_{21}e_1+a_{22}e_2+a_{23}e_3,e_3] -[a_{31}e_1+a_{32}e_2+a_{33}e_3,e_2]+[e_2,e_3]) \\
&=&[a_{21}e_1+a_{22}e_2+a_{23}e_3,\quad a_{31}e_1+a_{32}e_2+a_{33}e_3], \\
&&f([a_{31}e_1+a_{32}e_2+a_{33}e_3,e_1] -[a_{11}e_1+a_{12}e_2+a_{13}e_3,e_3]+[e_3,e_1]) \\
&=&[a_{31}e_1+a_{32}e_2+a_{33}e_3,\quad a_{11}e_1+a_{12}e_2+a_{13}e_3], \\
&&f([a_{11}e_1+a_{12}e_2+a_{13}e_3,e_2] -[a_{21}e_1+a_{22}e_2+a_{23}e_3,e_1]+[e_1,e_2]) \\
&=&[a_{11}e_1+a_{12}e_2+a_{13}e_3,\quad a_{21}e_1+a_{22}e_2+a_{23}e_3].
\end{eqnarray*}
}
Expanding the right hand sides, we have
\begin{eqnarray*}
f((a_{22}+a_{33}+ 1)e_1-a_{21}e_2-a_{31}e_3) &=&\twodet{a_{22}}{a_{23}}{a_{32}}{a_{33}}e_1 +\twodet{a_{23}}{a_{21}}{a_{33}}{a_{31}}e_2 +\twodet{a_{21}}{a_{22}}{a_{31}}{a_{32}}e_3, \\
f(-a_{12}e_1+(a_{11}+a_{33}+ 1)e_2-a_{32}e_3) &=&\twodet{a_{32}}{a_{33}}{a_{12}}{a_{13}}e_1 +\twodet{a_{33}}{a_{31}}{a_{13}}{a_{11}}e_2 +\twodet{a_{31}}{a_{32}}{a_{11}}{a_{12}}e_3, \\
f(-a_{13}e_1-a_{23}e_2+(a_{11}+a_{22}+ 1)e_3) &=&\twodet{a_{12}}{a_{13}}{a_{22}}{a_{23}}e_1 +\twodet{a_{13}}{a_{11}}{a_{23}}{a_{21}}e_2 +\twodet{a_{11}}{a_{12}}{a_{21}}{a_{22}}e_3.
\end{eqnarray*}
By the linearity of $f$, we can express these equations in the matrix equation
$$\threematrix{a_{22}+a_{33}+1}{-a_{12}}{-a_{13}}{-a_{21}}{a_{11}+a_{33}+1}{-a_{23}}{-a_{31}}{-a_{32}}{a_{11}+a_{22}+1} \threevec{f(e_1)}{f(e_2)}{f(e_3)} = (A^*)' \threevec{e_1}{e_2}{e_3}.$$
That is,
$$((\mathrm{tr}(A)+1)I_3-A')A\threevec{e_1}{e_2}{e_3}=(A^*)' \threevec{e_1}{e_2}{e_3}.$$
Since $\{e_1,e_2,e_3\}$ is a basis of $\mathrm{sl}(2,\complex)$, we obtain
$$((\mathrm{tr}(A)+1)I_3-A')A=(A^*)'.$$
This is Eq.~(\ref{eq:mateq}).
\smallskip

\noindent
(\ref{it:mat}) $\Longrightarrow$ (\ref{it:map}):
Reversing the above calculation, from Eq.~(\ref{eq:mateq}), we have
$$\threevec{f(\{ e_2, e_3\})}{f(\{ e_3, e_1\})}{f(\{ e_1, e_2\})}
=((\mathrm{tr}(A)+1)I_3-A')A\threevec{e_1}{e_2}{e_3}
=(A^*)' \threevec{e_1}{e_2}{e_3}
=\threevec{{[f(e_2),f(e_3)]}}{{[f(e_3),f(e_1)]}}{{[f(e_1),f(e_2)]}}.$$
So by Remark~\ref{rk:rb}.(\ref{it:rb2}), Eq.~(\ref{eq:RB1}) holds.
\smallskip

Finally, suppose there are linear maps $f$ and $g$ on $\mathrm{sl}(2,\complex)$ such that $x\circ y=[f(x),y]=[g(x),y], x, y\in \mathrm{sl}(2,\complex)$. Since the center of $\mathrm{sl}(2,\complex)$ is zero, we have $f(x)=g(x), x\in \mathrm{sl}(2,\complex)$. This proves the uniqueness of $f$.
\end{proof}

It is known~\cite{J} that an automorphism of $\mathrm{sl}(2,\complex)$ is given as a linear isomorphism: $X\mapsto AXA^{-1}$, $X\in \mathrm{sl}(2,\complex)$,
where $A$ is an invertible $2\times 2$ matrix. We need another description of such an automorphism in terms of $3\times 3$ matrices under a basis of the 3-dimensional Lie algebra $\mathrm{sl}(2,\complex)$. For the lack of a reference, we include a proof in the appendix.

\begin{lemma} A linear map on $\mathrm{sl}(2,\complex)$ is a Lie algebra automorphism if and only if its matrix with respect to the basis $\{e_1,e_2,e_3\}$ is in $\mathrm{SO}(3,\complex)$.\label{le:ag}
\end{lemma}

\begin{theorem}
\begin{enumerate}
\item\label{it:iff}
Two PostLie algebras $(\mathrm{sl}(2,\complex),[,],\circ)$ and $(\mathrm{sl}(2,\complex),[,],\star)$ on the Lie algebra $(\mathrm{sl}(2,\complex),[,])$ are isomorphic if and only if their matrices $A_\circ$ and $A_\star$ are congruent under
$\mathrm{SO}(3,\complex)$, that is, there is $T\in \mathrm{SO}(3,\complex)$ such that $A_\star=T'A_\circ T$.
\item
If $A$ is the matrix of a PostLie algebra, then all the matrices in its congruent class under $\mathrm{SO}(3,\complex)$ are matrices of PostLie algebras.
Thus $\mathrm{SO}(3,\complex)$ acts on the matrices of PostLie algebras on $(\mathrm{sl}(2,\complex),[,])$.
\label{it:full}
\item
The map that sends a solution of Eq.~(\ref{eq:mateq}) to its corresponding PostLie algebra in Theorem~\ref{thm:mat1} induces a bijection between congruent classes $($under $\mathrm{SO}(3,\complex)$$)$ of solutions of Eq.~(\ref{eq:mateq})
and isomorphic classes of PostLie algebra structures on $(\mathrm{sl}(2,\complex),[,])$.
\label{it:1-1}
\end{enumerate}
\label{thm:mat2}
\end{theorem}

\begin{proof}
(\ref{it:iff}).
($\Longrightarrow$)
According to the definition, two PostLie algebras $(\mathrm{sl}(2,\complex),[,],\circ)$ and $(\mathrm{sl}(2,\complex),[,],\star)$ are isomorphic on the Lie algebra $(\mathrm{sl}(2,\complex),[,])$ means that there is a linear isomorphism $\varphi: \mathrm{sl}(2,\complex)\to \mathrm{sl}(2,\complex)$ such that
\begin{eqnarray}
\varphi([x,y])&=&[\varphi(x),\varphi(y)],
\label{eq:mor1}
\\
\varphi(x\circ y)&=&\varphi(x)\star\varphi(y), \quad \forall x,y\in \mathrm{sl}(2,\complex).
\label{eq:mor2}
\end{eqnarray}

Let $f_\circ$ and $f_\star$ be the linear maps of the PostLie algebras $(\mathrm{sl}(2,\complex),[,],\circ)$ and $(\mathrm{sl}(2,\complex)$, $[,],\star)$.
By Eq. ~(\ref{eq:mor2}), we have
$$\varphi([f_\circ(x),y])=[f_\star(\varphi(x)),\varphi(y)], \quad \forall x,y\in \mathrm{sl}(2,\complex).$$
By Eq.~(\ref{eq:mor1}), we obtain
$$[\varphi(f_\circ(x)),\varphi(y)]=[f_\star(\varphi(x)),\varphi(y)], \quad \forall x,y\in \mathrm{sl}(2,\complex).$$
Since the center of $\mathrm{sl}(2,\complex)$ is zero, we have
$$\varphi(f_\circ(x))=f_\star(\varphi(x)), \quad \forall x\in \mathrm{sl}(2,\complex).$$
Thus $\varphi f_\circ=f_\star \varphi$, that is,
$A_\circ T=T A_\star$
for the matrix $T$ of $\varphi$ with respect to the basis $\{e_1,e_2,e_3\}$ of $\mathrm{sl}(2,\complex)$. Since $T$ is in $\mathrm{SO}(3,\complex)$ by Lemma~\ref{le:ag}, we have
$A_\star=T' A_\circ T.$
\smallskip

\noindent
($\Longleftarrow$)
Suppose there is $T\in \mathrm{SO}(3,\complex)$ such that $A_\star=T' A_\circ T$. Let $\psi$ be the linear operator on $\mathrm{sl}(2,\complex)$ whose matrix with respect to the basis $\{e_1,e_2,e_3\}$ is $T$. By Lemma~\ref{le:ag}, $\psi$ is an automorphism of the Lie algebra $\mathrm{sl}(2,\complex)$. Thus Eq.~(\ref{eq:mor1}) holds.
Furthermore, from $A_\star=T'A_\circ T$ we obtain $\psi f_\circ=f_\star \psi$. Thus we have
\begin{eqnarray*}
\psi(x\circ y)&=&\psi([f_\circ(x),y])=[\psi(f_\circ(x),\psi(y))]=[(\psi f_\circ)(x),\psi(y)]\\
&=&[(f_\star \psi)(x),\psi(y)]=[f_\star(\psi(x)),\psi(y)]=\psi(x)\star\psi(y),
\end{eqnarray*}
proving Eq.~(\ref{eq:mor2}).
\medskip

\noindent
(\ref{it:full}).
If $A$ is the matrix of a PostLie algebra on the Lie algebra $(\mathrm{sl}(2,\complex),[,])$, then A is a complex matrix satisfying Eq.~(\ref{eq:mateq}). Thus for any $T\in \mathrm{SO}(3,\complex)$, we have
$$T' A' T T'((\mathrm{tr}(A)+1)I_3-A)T=T' A^* T.$$
Since $T'=T^*$, this gives
$$(T' A T)'((\mathrm{tr}(A)+1)I_3-T' A T)=T^* A^* (T^*)'=T^* A^* (T')^*=(T'A T)^*,$$
showing that $B=T'AT$ also satisfies Eq.~(\ref{eq:mateq}). So $B$ is also the matrix of a PostLie algebra on $\mathrm{sl}(2,\complex)$.

\medskip

\noindent
(\ref{it:1-1}).
By Theorem~\ref{thm:mat1}, we have a bijective map  from the set of solutions of Eq.~(\ref{eq:mateq}) to the set of PostLie algebras on $(\mathrm{sl}(2,\complex),[,])$. By Item~(\ref{it:iff}), this bijective map induces a bijective map from the set of congruent classes (under the action of $\mathrm{SO}(3,\complex)$) of the solutions of Eq.~(\ref{eq:mateq}) to the set of isomorphic classes of PostLie algebras on $(\mathrm{sl}(2,\complex),[,])$.
\end{proof}

\section{Classification of the matrix solutions}
\label{sec:matclass}

According to Theorem~\ref{thm:mat2}.(\ref{it:1-1}), in order to prove our main Theorem~\ref{thm:class},
we only need to prove the following theorem on congruent classes of solutions of Eq.~(\ref{eq:mateq}).

\begin{theorem}
A complete list of representatives $A$ of congruent classes for the solutions of Eq.~(\ref{eq:mateq}) is given as follows:
\begin{eqnarray*}
& &\threematrix{0}{0}{0}{0}{0}{0}{0}{0}{0}; \quad
\threematrix{-1}{0}{0}{0}{-1}{0}{0}{0}{-1}; \quad  \threematrix{-1}{0}{0}{0}{-\frac{1+\sqrt{-1}} {2}}{\frac{\sqrt{-1}-1}{2}}{0} {-\frac{1+\sqrt{-1}}{2}}{\frac{\sqrt{-1}-1}{2}};\ \\
& &\threematrix{k}{0}{0}{0}{-\frac{1}{2}} {\frac{\sqrt{-1}}{2}}{0}{-\frac{\sqrt{-1}}{2}} {-\frac{1}{2}}, k\in \complex; \quad
\threematrix{-\frac{1}{2}+\sqrt{-1}} {1-\frac{\sqrt{-1}}{2}}{0}{1+\frac{\sqrt{-1}}{2}}{-\frac{1}{2}-\sqrt{-1}}{0}{0}{0}{0}.
\end{eqnarray*}
\label{thm:matclass}
\end{theorem}

The proof of this theorem will be presented in this section.
After discussing a preparatory result on congruent classes of complex symmetric matrices, we will divide our proof into the three cases when the rank of $A$ is three, two or one. Since the case when the rank of $A$ is zero gives us the trivial solution, we will not discuss it further.

The next result is essentially due to~\cite[Chapter XI, Corollary 2]{Ga}. We modify it in its general form and spell out the details in the dimension three case to fit the application in this paper.
\begin{proposition}
Consider the $k\times k$ complex matrix
$$D_k:=\left (\begin{matrix}
0& 1& 0 &\ldots &0 & 0\\
1& 0& 1 &\ldots &0 & 0 \\
0& 1& 0 &\ldots &0 & 0 \\
\vdots & \vdots & \vdots &\ddots &\vdots& \vdots\\
0 & 0 & 0 &\ldots & 0 & 1\\
0 & 0 & 0& \ldots & 1 & 0
\end{matrix} \right )
+ \sqrt{-1}\,\left (\begin{matrix}
0&0 &\cdots & 0&1& 0\\
0&0 &\cdots & 1 & 0 & -1\\
0&0 &\cdots & 0 & -1 & 0 \\
\vdots & \vdots& \ddots & \vdots &\vdots&\vdots \\
1 & 0 & \cdots & 0 & 0 & 0 \\
0 & -1 & \cdots & 0 & 0 & 0
\end{matrix}
\right )
$$
\begin{enumerate}
\item
For a complex symmetric $n\times n$ matrix $A$ with elementary factors $(\lambda-\lambda_1)^{k_1},\cdots, (\lambda-\lambda_t)^{k_t}$, $k_1+\cdots+k_t=n$, there exists $T\in O(n,\complex)$ such that
\begin{equation}
TAT^{-1}=P:=diag(\lambda_1 I_{k_1}+ D_{k_1},\lambda_2 I_{k_2}+D_{k_2},\cdots,\lambda_t I_{k_t}+D_{k_t}).
\label{eq:matP}
\end{equation}
\label{it:csm}
\item
When $n$ is odd, the above matrix $T$ can be chosen to be in $SO(n,\complex)$.
\label{it:csmo}
\item
For each $3\times 3$ complex symmetric matrix $A$, there exists $T\in SO(3,\complex)$ and a unique $P$ is the following list such that
$P=TAT^{-1}$.
\begin{enumerate}
\item
When $\mathrm{r}(A)=3$,
\begin{eqnarray}
& \left (\begin{matrix}\lambda_1&0&0\\ 0&\lambda_2&0\\
0&0&\lambda_3\end{matrix}\right);  \threematrix{\lambda_1}{0}{0}{0}{\lambda_2 +\sqrt{-1}\hspace{-.25cm}}{1}{0}{1}{\hspace{-.25cm}\lambda_2-\sqrt{-1}};
&
\notag \\
&\threematrix{\lambda}{1+\sqrt{-1}}{0} {1+\sqrt{-1}}{\lambda}{1- \sqrt{-1}}{0}{1-\sqrt{-1}}{\lambda}. & \label{eq:3list3}
\end{eqnarray}
\item
When $\mathrm{r}(A)=2$,
\begin{eqnarray}
&\threematrix{\lambda_1}{0}{0}{0} {\lambda_2}{0}{0}{0}{0}; \ \threematrix{\lambda+\sqrt{-1}\hspace{-.4cm}}{1}{0}{1} {\hspace{-.4cm}\lambda-\sqrt{-1}}{0}{0}{0}{0};&
\notag \\
&\threematrix{\lambda}{0}{0}{0}{\sqrt{-1}} {1}{0}{1}{-\sqrt{-1}};
\
\threematrix{0}{\hspace{-.25cm}1+\sqrt{-1}}{0} {1+\sqrt{-1}\hspace{-.25cm}}{0}{\hspace{-.25cm}1-\sqrt{-1}} {0}{\hspace{-.25cm}1-\sqrt{-1}}{0}. &
\label{eq:3list2}
\end{eqnarray}
\item
When $\mathrm{r}(A)=1$,
\begin{eqnarray}
\threematrix{\lambda}{0}{0}{0}{0} {0}{0}{0}{0};
\quad
\threematrix{\sqrt{-1}}{1}{0}{1} {-\sqrt{-1}}{0}{0}{0}{0} \label{eq:3list1}.
\end{eqnarray}
\item
When $\mathrm{r}(A)=0,$
\begin{eqnarray}
\threematrix{0}{0}{0}{0}{0} {0}{0}{0}{0}.
\label{eq:3list0}
\end{eqnarray}
\end{enumerate}
Here all constants are non-zero.
\label{it:csm3}
\end{enumerate}
\label{pp:clale}
\end{proposition}

\begin{proof}
(\ref{it:csm}).
By~\cite[Chapter XI, Corollary 2]{Ga}, for a complex symmetric matrix $A$ with the elementary factors as in the proposition, there exists $T\in O(n,\complex)$ such that \begin{equation}
TAT^{-1}=P:=diag(\lambda_1 I_{k_1}+ \frac{1}{2}D_{k_1},\lambda_2 I_{k_2}+\frac{1}{2}D_{k_2},\cdots,\lambda_t I_{k_t}+\frac{1}{2}D_{k_t}).
\label{eq:matP1}
\end{equation}

Note that
$\frac{1}{2}A$ is a complex symmetric matrix whose
elementary factor is $(\lambda-\frac{1}{2}\lambda_1)^{k_1}, \cdots, (\lambda-\frac{1}{2}\lambda_t)^{k_t}$. Applying
the above result in~\cite{Ga} to $\frac{1}{2}A$, there is $T\in O(n,\complex)$ such that
$$T\frac{1}{2}AT^{-1}=diag(\frac{1}{2}\lambda_1 I_{k_1} +\frac{1}{2}D_{k_1},\cdots, \frac{1}{2}\lambda_tI_{k_t}+\frac{1}{2}D_{k_t}).$$
Hence we have
$$TAT^{-1}=\diag(\lambda_1 I_{k_1} +D_{k_1},\cdots,\lambda_{t} I_{k_t}+D_{k_t}).$$

\noindent
(\ref{it:csmo}).
When $n$ is odd, we can get the matrix $T\in O(n,\mathbb C)$ to be a
matrix $S\in SO(n,\mathbb C)$ by keeping  $T$ if $\det~ T=1$ and replace $T$ by $-T$ if $\det~ T=-1$.

\noindent
(\ref{it:csm3}).
This part is a detailed enumeration of Item~(\ref{it:csmo} except the uniqueness of $P$ which follows since different matrices in the list of Eqs.~(\ref{eq:3list3}) -- (\ref{eq:3list0}) have different Jordan canonical forms.
\end{proof}

\subsection{Case 1 of Theorem~\ref{thm:matclass}: the rank of $A$ is three}

In this case, $A$ is invertible. Then by Eq.~(\ref{eq:mateq}), we obtain
$$(\mathrm{tr}(A)+1)I_3-A=(A')^{-1}A^*=(A^{-1})'A^{-1}\det A.$$
Then
$$A=(\mathrm{tr}(A)+1)I_3-(A^{-1})'A^{-1}\det A.$$
So A is a symmetric matrix. By Proposition~\ref{pp:clale}, we can assume that $A$ is one of the three matrices in Eq.~(\ref{eq:3list3}).

\subsubsection{Case 1.1:} $A=\left (\begin{matrix}\lambda_1&0&0\\ 0&\lambda_2&0\\
0&0&\lambda_3\end{matrix}\right), \lambda_1\lambda_2\lambda_3\neq 0.$

Applying Eq.~(\ref{eq:mateq}) to $A$ and comparing the entries on the main diagonal of the matrices on the two sides, we have
\begin{equation*}
\lambda_1(\lambda_2+\lambda_3+1)=\lambda_2\lambda_3,\quad
\lambda_2(\lambda_3+\lambda_1+1)=\lambda_3\lambda_1,\quad
\lambda_3(\lambda_1+\lambda_2+1)=\lambda_1\lambda_2.
\end{equation*}
Adding two of equations at a time, we obtain
\begin{equation*}
2\lambda_1\lambda_2+\lambda_1+\lambda_2=0, \quad 2\lambda_2\lambda_3+\lambda_2+\lambda_3=0, \quad
2\lambda_3\lambda_1+\lambda_3+\lambda_1=0.
\end{equation*}
Since $\lambda_1\lambda_2\lambda_3\neq 0$, we get
\begin{equation*}
\frac{1}{\lambda_1}+\frac{1}{\lambda_2}+2=0, \quad
\frac{1}{\lambda_2}+\frac{1}{\lambda_3}+2=0, \quad
\frac{1}{\lambda_3}+\frac{1}{\lambda_1}+2=0.
\end{equation*}
Solving this linear system, we obtain $\frac{1}{\lambda_1}=\frac{1}{\lambda_2}=\frac{1}{\lambda_3}=-1$. That is $\lambda_1=\lambda_2=\lambda_3=-1$.
Therefore, $$A=\threematrix{-1}{0}{0}{0}{-1}{0}{0}{0}{-1}.$$
This is indeed a solution of Eq.~(\ref{eq:mateq}).

\subsubsection{Case 1.2:} $A=\threematrix{\lambda_1}{0}{0}{0}{\lambda_2 +\sqrt{-1}}{1}{0}{1}{\lambda_2-\sqrt{-1}}, \lambda_1\lambda_2\neq0$.

Apply Eq.~(\ref{eq:mateq}) to $A$ and comparing entries, we have
\begin{equation*}
2\lambda_1\lambda_2+\lambda_1=(\lambda_2)^2, \quad
(2\lambda_1+1)\sqrt{-1}+(\lambda_2)^2+\lambda_2=0, \quad
2\lambda_1+1=0.
\end{equation*}
From the third equation, we obtain $\lambda_1=-\frac{1}{2}\smallskip.$ Substituting it into the second equation, we have $\lambda_2^2+\lambda_2=0$.
So $\lambda_2$ is $0$ or $-1$, both contradicting the first equation. Thus the equation set does not have any solution and this case does not give any solution of Eq.~(\ref{eq:mateq}).

\subsubsection{Case 1.3:} $A=\threematrix{\lambda}{1+\sqrt{-1}}{0} {1+\sqrt{-1}}{\lambda}{1- \sqrt{-1}}{0}{1-\sqrt{-1}}{\lambda}, \lambda\neq0.$

Applying Eq.~(\ref{eq:mateq}) to $A$ and comparing the $(1,3)$-entries of the two sides, we get $0-2=2$, which is a contradiction. Therefore this case does not give any solution of Eq.~(\ref{eq:mateq}).

\subsection{Case 2 of Theorem~\ref{thm:matclass}: the rank of $A$ is two}
In this case, $A$ is not necessarily symmetric. But $A'A$ is still a symmetric matrix. We will use this observation to relate Eq.~(\ref{eq:mateq}) for $A$ to an equation for $A'A$.

Let $A$ be a solution of Eq.~(\ref{eq:mateq}).
Multiplying $A$ to two sides of  Eq.~(\ref{eq:mateq}) from the right, we have
$A'((\mathrm{tr}{A}+1)I_3-A)A=\det A I_3=0$.
Thus
\begin{equation}
(\mathrm{tr}{A}+1)A'A=A'AA.
\label{eq:21a}
\end{equation}
Therefore, for a solution $A$ of Eq.~(\ref{eq:mateq}), $A'A$ is symmetric and satisfies Eq.~(\ref{eq:21a}).

Furthermore, since $\mathrm{r}(A)=2$ and $\mathrm{r}(A'A)\leq \mathrm{r}(A)$, $\mathrm{r}(A'A)$ is $0$, or $1$ or $2$. However, $\mathrm{r}(A'A)\ne 0$. Otherwise, $\mathrm{r}(A')=2$ would be
the dimension of the solution space of $AX=0$, which is $3-\mathrm{r}(A)=1$, a contradiction. Thus $\mathrm{r}(A'A)=1$ or 2.
Then by Proposition~\ref{pp:clale}, there is $T\in \mathrm{SO}(3,\complex)$ such that $P:=TA'AT'$ is one of the six matrices in Eqs.~(\ref{eq:3list2}) and (\ref{eq:3list1}).

By Theorem~\ref{thm:mat2}, $B:=TAT'$ is also a solution of Eq.~(\ref{eq:mateq}) that is congruent to $A$. Further, multiplying $T$ (resp. $T'$) to the left (resp. right) hand side of Eq.~(\ref{eq:21a}), we find that $B'B=TA'AT'=P$ satisfies Eq.~(\ref{eq:21a}) as well and is congruent to $A'A$.

To summarize, in order to find solutions of Eq.~(\ref{eq:mateq}) of rank 2 up to congruent by $\mathrm{SO}(3,\complex)$, we only need to consider every solution $A$ of Eq.~(\ref{eq:mateq}) such that $A'A$ is one of the six matrices in Eqs.~(\ref{eq:3list2}) and (\ref{eq:3list1}), and satisfies Eq.~(\ref{eq:21a}). We now consider the corresponding six cases separately.

\subsubsection{Case 2.1:} $A'A=\threematrix{\lambda_1}{0}{0}{0} {\lambda_2}{0}{0}{0}{0}, \lambda_1\lambda_2\neq0.$

Substituting $A'A$ into Eq.~(\ref{eq:21a}) and comparing entries in the first two rows, we have
\begin{equation}
a_{12}=a_{13}=a_{21}=a_{23}=0,\quad  a_{11}=a_{22}=\mathrm{tr}(A)+1.
\label{eq:case2.1}
\end{equation}
If $\mathrm{tr}(A)+1=0$, then the first two rows of $A$ are zero. Then $\mathrm{r}(A)\leqslant1$, which is a contradiction to our assumption. So $a_{11}=a_{22}=\mathrm{tr}(A)+1\neq0$. Then, from $\mathrm{r}(A)=2$, we get $a_{33}=0$. Substituting it into Eq.~(\ref{eq:mateq}) and comparing the $(3,3)$-entries of the two sides, we get $0-0=a_{11}a_{22}\neq0$, which is a contradiction. Therefore this case does not give any solution of Eq.~(\ref{eq:mateq}).

\subsubsection{
Case 2.2:} $A'A=\threematrix{\lambda+\sqrt{-1}}{1}{0}{1} {\lambda-\sqrt{-1}}{0}{0}{0}{0}, \lambda\neq0$.

Substituting $A'A$ into Eq.~(\ref{eq:21a}) and comparing entries in the first two rows, we again have Eq.~(\ref{eq:case2.1}). Therefore, as in Case~2.1, the current case does not give any solution of Eq.~(\ref{eq:mateq}).

\subsubsection{Case 2.3:} $A'A=\threematrix{\lambda}{0}{0}{0}{\sqrt{-1}} {1}{0}{1}{-\sqrt{-1}}, \lambda\neq0$.

Substituting the form of $A'A$ into Eq.~(\ref{eq:21a}), we obtain
\begin{equation}
(\mathrm{tr}(A)+1)\threematrix{\lambda}{0}{0}{0} {\sqrt{-1}}{1}{0}{1}{-\sqrt{-1}}=
\threematrix{\lambda}{0}{0}{0} {\sqrt{-1}}{1}{0}{1}{-\sqrt{-1}}A.
\label{eq:23a}
\end{equation}
Since $\lambda\neq0$, we have
$$a_{11}=\mathrm{tr}(A)+1, \quad a_{12}=a_{13}=0.$$
On the other hand, by the assumption of the form of $A'A$ in this case, we obtain
$$\twomatrix{a_{22}}{a_{32}}{a_{23}}{a_{33}}\twomatrix{a_{22}}{a_{23}}{a_{32}}{a_{33}}
=\twomatrix{\sqrt{-1}}{1}{1}{-\sqrt{-1}}.$$
Hence
\begin{equation}
(a_{22})^2+(a_{32})^2=\sqrt{-1},\quad
(a_{23})^2+(a_{33})^2=-\sqrt{-1}, \quad
\det \twomatrix{a_{22}}{a_{23}}{a_{32}}{a_{33}}=0.
\label{eq:23bc}
\end{equation}
From the third equation in Eq.~(\ref{eq:23bc}) and $a_{12}=a_{13}=0$, we find that the first row of $A^*$ is 0.
Thus from Eq.~(\ref{eq:mateq}), we obtain
$$(\mathrm{tr}(A)+1)\threematrix{a_{11}}{a_{21}}{a_{31}}{0}{a_{22}}{a_{32}}{0}{a_{23}}{a_{33}}
-\threematrix{\lambda}{0}{0}{0}{\sqrt{-1}}{1}{0}{1}{-\sqrt{-1}}=
\threematrix{0}{0}{0}{(A^*)_{21}}{(A^*)_{22}}{(A^*)_{23}}{(A^*)_{31}}{(A^*)_{32}}{(A^*)_{33}}.$$
Comparing the (1, 1)-entries, we have $(\mathrm{tr}(A)+1)a_{11}=(\mathrm{tr}(A)+1)^2=\lambda\neq0$. Then $a_{21}=a_{31}=0.$
So from Eq.~(\ref{eq:23a}), we get
$$(\mathrm{tr}(A)+1)\twomatrix{\sqrt{-1}}{1}{1}{-\sqrt{-1}}=
\twomatrix{\sqrt{-1}}{1}{1}{-\sqrt{-1}}\twomatrix{a_{22}}{a_{23}}{a_{32}}{a_{33}}.$$
Thus we obtain
\begin{equation}
(\mathrm{tr}(A)+1)\sqrt{-1}=\sqrt{-1}a_{22}+a_{32},
\quad
-(\mathrm{tr}(A)+1)\sqrt{-1}=a_{23}-\sqrt{-1}a_{33}.
\label{eq:23de}
\end{equation}
It is easy to derive the solutions of  Eq.~(\ref{eq:23bc}) and  Eq.~(\ref{eq:23de}):
\begin{eqnarray*}
a_{22}&=&\frac{1}{2}(\mathrm{tr}(A)+1+\frac{\sqrt{-1}}{\mathrm{tr}(A)+1}); \quad a_{32}=\frac{1}{2}((\mathrm{tr}(A)+1)\sqrt{-1}+\frac{1}{\mathrm{tr}(A)+1}).\\
a_{23}&=&\frac{1}{2}(\frac{1}{\mathrm{tr}(A)+1}-(\mathrm{tr}(A)+1)\sqrt{-1}); \quad a_{33}=\frac{1}{2}((\mathrm{tr}(A)+1)-\frac{\sqrt{-1}}{\mathrm{tr}(A)+1}).
\end{eqnarray*}
Hence $a_{22}+a_{33}=\mathrm{tr} (A)+1$. On the other hand, since $\mathrm{tr}(A)+1=a_{11}$, we have $a_{22}+a_{33}=-1$. Therefore $\mathrm{tr}(A)+1=-1$. So
$$A=\threematrix{-1}{0}{0}{0}{-\frac{1+\sqrt{-1}}{2}}{\frac{\sqrt{-1}-1}{2}}{0}{-\frac{1+\sqrt{-1}}{2}}{\frac{\sqrt{-1}-1}{2}}.$$
It is straightforward to check that it satisfies Eq.~(\ref{eq:mateq}).

\subsubsection{Case 2.4:} $A'A=\threematrix{0}{1+\sqrt{-1}}{0} {1+\sqrt{-1}}{0}{1-\sqrt{-1}}{0}{1-\sqrt{-1}} {0}$.

Substituting $A'A$ into Eq.~(\ref{eq:21a}), we obtain
\begin{equation}a_{21}=a_{23}=0,\quad a_{22}=\mathrm{tr}(A)+1,\quad (a_{12})^2+(a_{32})^2=0.
\label{eq:24a}
\end{equation}
On the other hand, comparing the $(2,2)$-entries on both sides of $A'A$ in its assumed form in this case, we obtain
$$(a_{12})^2+(a_{22})^2+(a_{32})^2=0.$$ So $a_{22}=0$. Thus $\mathrm{tr}(A)+1=0$.
Therefore by  Eq.~(\ref{eq:21a}) in this case again,  we obtain
$$(1+\sqrt{-1},1-\sqrt{-1})\left(\begin{array}{ccc}{a_{11}}&{a_{12}}&{a_{13}}\\{a_{31}}&{a_{32}}&{a_{33}}\end{array}\right)=0.\smallskip$$
Note that $(a_{21},a_{22},a_{23})=0$. So $\mathrm{r}(A)=1$, which is a contradiction to our assumption. Therefore this case does not give any solution of Eq.~(\ref{eq:mateq}).

\subsubsection{Case 2.5:} $A'A=\threematrix{\lambda}{0}{0}{0}{0} {0}{0}{0}{0}, \lambda\neq0$.

Substituting the form of $A'A$ into Eq.~(\ref{eq:21a}), we obtain
$$a_{11}=\mathrm{tr}(A)+1, \quad a_{12}=a_{13}=0.$$
On the other hand, by the assumption of the form of $A'A$ in this case, we obtain
\begin{equation}
\twomatrix{a_{22}}{a_{32}}{a_{23}}{a_{33}}\twomatrix{a_{22}}{a_{23}}{a_{32}}{a_{33}}=0.
\label{eq:25a}
\end{equation}
From Eq.~(\ref{eq:mateq}), we have
\begin{equation}
(\mathrm{tr}(A)+1)\threematrix{a_{11}}{a_{21}}{a_{31}}{0}{a_{22}}{a_{32}}{0}{a_{23}}{a_{33}}
-\threematrix{\lambda}{0}{0}{0}{0}{0}{0}{0}{0}=
\threematrix{0}{0}{0}{(A^*)_{21}}{(A^*)_{22}}{(A^*)_{23}}{(A^*)_{31}}{(A^*)_{32}}{(A^*)_{33}}.
\label{eq:25b}
\end{equation}
Comparing the $(1,1)$-entries of the two sides, we have $(\mathrm{tr}(A)+1)a_{11}=(\mathrm{tr}(A)+1)^2=\lambda\neq0\smallskip$. So $a_{11}\neq0$.
Moreover, comparing the $(1,2)$ and (1,3)-entries of the two sides, we see that $a_{21}=a_{31}=0$.
Thus Eq.~(\ref{eq:25b}) becomes
$$(\mathrm{tr}(A)+1)\threematrix{a_{11}}{0}{0}{0}{a_{22}}{a_{32}}{0}{a_{23}}{a_{33}}
-\threematrix{\lambda}{0}{0}{0}{0}{0}{0}{0}{0}=\threematrix{0}{0}{0}{0}{a_{11}a_{33}}{-a_{11}a_{23}}{0}{-a_{11}a_{32}}{a_{11}a_{22}}$$
Since $a_{11}\neq0$, we obtain
$$a_{22}=a_{33}, \quad  a_{23}=-a_{32}.$$ Since $a_{11}=\mathrm{tr}(A)+1$, we have $a_{22}+a_{33}=-1$. So $a_{22}=a_{33}=-\frac{1}{2}$. Substituting them into Eq.~(\ref{eq:25a}) gives
$$A=\threematrix{a_{11}}{0}{0}{0}{-\frac{1}{2}}{\frac{\sqrt{-1}}{2}}{0}{-\frac{\sqrt{-1}}{2}}{-\frac{1}{2}}\;\;{\rm and}\;\;A=\threematrix{a_{11}}{0}{0}{0}{-\frac{1}{2}}{{-\frac{\sqrt{-1}}{2}}}{0}{\frac{\sqrt{-1}}{2}}{-\frac{1}{2}}.$$
However, since
$$\threematrix{-1}{0}{0}{0}{0}{-1}{0}{-1}{0}
\threematrix{a_{11}}{0}{0}{0}{-\frac{1}{2}}{-\frac{\sqrt{-1}}{2}}{0}{\frac{\sqrt{-1}}{2}}{-\frac{1}{2}}
\threematrix{-1}{0}{0}{0}{0}{-1}{0}{-1}{0}=
\threematrix{a_{11}}{0}{0}{0}{-\frac{1}{2}}{\frac{\sqrt{-1}}{2}}{0}{-\frac{\sqrt{-1}}{2}}{-\frac{1}{2}}, $$
the above two matrices are orthogonal congruent. So every matrix in this case is orthogonal congruent to one of
$$\threematrix{a_{11}}{0}{0}{0}{-\frac{1}{2}} {\frac{\sqrt{-1}}{2}}{0}{-\frac{\sqrt{-1}}{2}} {-\frac{1}{2}},\quad a_{11}\neq0,$$
which satisfies Eq.~(\ref{eq:mateq}) for any $a_{11}\ne 0$. Thus in this case, we obtain a parameterized family of solutions.
Since the trace of a matrix is preserved by any orthogonal congruent operator, the matrices with different values of $a_{11}$ are not congruent.
Therefore different matrices in the family are in different congruency classes.

\subsubsection{Case 2.6:} $A'A=\threematrix{\sqrt{-1}}{1}{0}{1} {-\sqrt{-1}}{0}{0}{0}{0}$.

Substituting $A'A$ into Eq.~(\ref{eq:21a}), we obtain
\begin{equation}
(\mathrm{tr}(A)+1)\threematrix{\sqrt{-1}}{1}{0}{1}{-\sqrt{-1}}{0}{0}{0}{0}=
\threematrix{\sqrt{-1}}{1}{0}{1}{-\sqrt{-1}}{0}{0}{0}{0}A.
\label{eq:26a}
\end{equation}

If $(\mathrm{tr}(A)+1)=0$, then
\begin{equation}
a_{11}=\sqrt{-1}a_{21},\quad a_{12}=\sqrt{-1}a_{22},\quad a_{13}=\sqrt{-1}a_{23}.
\label{eq:26b}
\end{equation}
So$$(A^*)_{11}=-\sqrt{-1}\twodet{a_{12}}{a_{13}}{a_{32}}{a_{33}}=-\sqrt{-1}(A^*)_{12}.$$
However, from Eq.~(\ref{eq:mateq}), we obtain $-A'A=A^*$. So $(A^*)_{11}=\sqrt{-1}(A^*)_{12}$ which is a contradiction. Therefore this case does not give  a solution of Eq.~(\ref{eq:mateq}).

If $(\mathrm{tr}(A)+1)\neq0$, then by Eq.(\ref{eq:26a}), we obtain $a_{13}=\sqrt{-1}a_{23}.$
 On the other hand, by the assumption of the form of $A'A$ in this case, we obtain
$(a_{13})^2+(a_{23})^2+(a_{33})^2=0.$
Hence $a_{33}=0$. Substituting them into Eq.~(\ref{eq:mateq}) and comparing the $(3,3)$-entries, we have
$a_{11}a_{22}-a_{12}a_{21}=0.$
So we have
\begin{align*}
(\sqrt{-1}a_{11}+a_{21})a_{22}-(\sqrt{-1}a_{12}+a_{22})a_{21}=0; \\ (\sqrt{-1}a_{11}+a_{21})a_{12}-(\sqrt{-1}a_{12}+a_{22})a_{11}=0.
\end{align*}
By Eq.~(\ref{eq:26a}) again, we obtain
$$\sqrt{-1}a_{11}+a_{21}=\sqrt{-1}(\sqrt{-1}a_{12}+a_{22})=(\mathrm{tr}(A)+1)\sqrt{-1}.$$
So
$$
(\mathrm{tr}(A)+1)(\sqrt{-1}a_{22}-a_{21})=0; \quad (\mathrm{tr}(A)+1)(\sqrt{-1}a_{12}-a_{11})=0.$$
Hence
\begin{align*}
a_{21}=\sqrt{-1}a_{22};\quad a_{11}=\sqrt{-1}a_{12}.
\end{align*}
By the form of $A'A$ again, we obtain
\begin{align*}
&(a_{11})^2+(a_{21})^2+(a_{31})^2=\sqrt{-1};
\quad
a_{11}a_{12}+a_{21}a_{22}+a_{31}a_{32}=1;
\\
&(a_{12})^2+(a_{22})^2+(a_{32})^2=-\sqrt{-1}.
\end{align*}
Hence we have
$$a_{31}a_{32}+\sqrt{-1}(a_{31})^2=0,\quad (a_{32})^2+(a_{31})^2=0.$$
So $a_{31}=\sqrt{-1}a_{32}$. Thus the last row of $A^*$ is 0. Furthermore from the last row of Eq.~(\ref{eq:mateq}), we obtain
$a_{13}=a_{23}=a_{33}=0.$ Then $\mathrm{r}(A)=1$, which is a contradiction to our assumption. Therefore this case does not give  any solution of Eq.~(\ref{eq:mateq}).

\subsection{Case 3 of Theorem~\ref{thm:matclass}: the rank of $A$ is one}

In this case, $A^*=0$. So Eq.~(\ref{eq:mateq}) becomes
\begin{equation}
(\mathrm{tr}(A)+1)A'=A'A.
\label{eq:3a}
\end{equation}
There are the following six subcases, including two subcases where $A$ is symmetric and four cases where $A$ is not symmetric.

\subsubsection{$A$ is symmetric}

Since $\mathrm{r}(A)=1$, by Proposition~\ref{pp:clale}, we have the following two subcases.
\begin{enumerate}
\item[\em Case 3.1: \hspace{-.6cm}] $\quad A=\threematrix{\lambda}{0}{0}{0}{0}{0}{0}{0}{0}, \lambda\neq0.$

Then from Eq.~(\ref{eq:3a}), we obtain  $(\lambda+1)\lambda=(\lambda)^2$.  So $\lambda=0$ which is a contradiction to our assumption. Thus this case does not give  a solution of Eq.~(\ref{eq:mateq}).

\item[\em Case 3.2: \hspace{-.6cm}] $\quad A=\threematrix{\sqrt{-1}}{1}{0}{1}{-\sqrt{-1}}{0}{0}{0}{0}$

Note that $A'A=0$. Combining it with Eq.~(\ref{eq:3a}), we have $(\mathrm{tr}(A)+1)A'=0$. Since $\mathrm{tr}(A)\neq-1$, we have $\mathrm{tr}(A)+1\neq0$. Then $A'=0$. Thus $A=0$, which is
a contradiction again. Therefore this case does not give a solution of Eq.~(\ref{eq:mateq}).
\end{enumerate}

\subsubsection{
$A$ is not symmetric}
In this case, we apply a strategy similar to Case 2 by relating $A$ to its symmetrizer $\frac{1}{2}(A+A')$.

First note that if $\mathrm{tr}(A)+1\neq0$, then $A$ is symmetric. So by our assumption, we obtain $\mathrm{tr}(A)+1=0$.
Then by Eq.~(\ref{eq:3a}), we also have $A'A=0$.

Let $A$ be a solution of Eq.~(\ref{eq:mateq}). Since $\frac{1}{2}(A+A')$ is symmetric, by Proposition~\ref{pp:clale}.(\ref{it:csm3}), there is $T\in \mathrm{SO}(3,\complex)$ such that $T\frac{1}{2}(A+A')T'$ is one of the matrices in Proposition~\ref{pp:clale}.(\ref{it:csm3}). By Theorem~\ref{thm:mat2}, $TAT'$ is a solution of Eq.~(\ref{eq:mateq}) that is congruent to $A$ and its symmetrizer $\frac{1}{2}(TAT'+(TAT')')=T\frac{1}{2}(A+A')T'$ is one of the matrices in Proposition~\ref{pp:clale}.(\ref{it:csm3}).

Therefore, to find solutions $A$ of Eq.~(\ref{eq:mateq}) with $\mathrm{r}(A)=1$ that is not symmetric, we only need to find from those $A$ such that $\frac{1}{2}(A+A')$ is from the matrices in Proposition~\ref{pp:clale}.(\ref{it:csm3}).

Furthermore, since $\mathrm{r}(A)=1$, we can suppose
$A=(\alpha_1,\alpha_2,\alpha_3)'\cdot (\beta_1,\beta_2,\beta_3), $ where not all $\alpha_i\in \complex$ are zero and not all $\beta_j\in\complex$ are zero.
Then
$$\frac{1}{2}(A+A')=\threematrix{\alpha_1\beta_1}{\frac{1}{2}(\alpha_1\beta_2+\alpha_2\beta_1)}
{\frac{1}{2}(\alpha_1\beta_3+\alpha_3\beta_1)}{\frac{1}{2}(\alpha_1\beta_2+\alpha_2\beta_1)}
{\alpha_2\beta_2}{\frac{1}{2}(\alpha_2\beta_3+\alpha_3\beta_2)}
{\frac{1}{2}(\alpha_1\beta_3+\alpha_3\beta_1)}{\frac{1}{2}(\alpha_2\beta_3+\alpha_3\beta_2)}{\alpha_3\beta_3}.\smallskip$$
If $\mathrm{r}(\frac{1}{2}(A+A'))<2$, then all the $2\times 2$ subdeterminants are 0.
Thus we have
\begin{align*}
(\alpha_1\beta_1)(\alpha_2\beta_2)-(\frac{1}{2}(\alpha_1\beta_2+\alpha_2\beta_1))^2&=0;\quad (\alpha_2\beta_2)(\alpha_3\beta_3)-(\frac{1}{2}(\alpha_2\beta_3+\alpha_3\beta_2))^2=0;\\
(\alpha_3\beta_3)(\alpha_1\beta_1)-(\frac{1}{2}(\alpha_3\beta_1+\alpha_1\beta_3))^2&=0.
\end{align*}
They simplify to
$$\alpha_1\beta_2=\alpha_2\beta_1,\quad \alpha_2\beta_3=\alpha_3\beta_2,\quad \alpha_3\beta_1=\alpha_1\beta_3.$$
Therefore $A$ is symmetric, which is a contradiction to our assumption. Thus  $\mathrm{r}(\frac{1}{2}(A+A'))\geqslant2\smallskip$.
On the other hand, by basic linear algebra, we have
$$\mathrm{r}(\frac{1}{2}(A+A'))=\mathrm{r}(A+A')\leq\mathrm{r}(A)+\mathrm{r}(A').$$
Hence $\mathrm{r}(\frac{1}{2}(A+A'))=2$. Thus by Proposition~\ref{pp:clale}.(\ref{it:csm3}), we only need to consider the following four cases:
\medskip

\noindent
{\em Case 3.3: } $\frac{1}{2}(A+A') =\threematrix{\lambda_1}{0}{0}{0} {\lambda_2}{0}{0}{0}{0}, \lambda_1\lambda_2\neq0.$

In this case, $$a_{11}=\lambda_1\neq0,\quad a_{12}+a_{21}=0,\quad a_{13}+a_{31}=0, \quad a_{33}=0.$$
Suppose $a_{21}=k a_{11}$. Then from $A'A=0$, we obtain
\begin{equation}
A=a_{11} \threematrix{1}{-k}{\mp\sqrt{-(k^2+1)}}{k}{-k^2}{\mp k\sqrt{-(k^2+1)}}{\pm\sqrt{-(k^2+1)}}{\mp k\sqrt{-(k^2+1)}}{k^2+1}.
\label{eq:33a}
\end{equation}
Since $a_{33}=0$, we have $k^2+1=0$, that is, $k=\pm\sqrt{-1}$. Moreover, since $\mathrm{tr}(A)=-1$, $a_{11}=-\frac{1}{2}$. Therefore
$$A=\threematrix{-\frac{1}{2}}{\frac{\sqrt{-1}}{2}}{0}{-\frac{\sqrt{-1}}{2}}{-\frac{1}{2}}{0}{0}{0}{0}\;\;{\rm or}\;\;
A=\threematrix{-\frac{1}{2}}{-\frac{\sqrt{-1}}{2}}{0}{\frac{\sqrt{-1}}{2}}{-\frac{1}{2}}{0}{0}{0}{0}.$$
However, since
$$\threematrix{0}{-1}{0}{-1}{0}{0}{0}{0}{-1}
\threematrix{-\frac{1}{2}}{-\frac{\sqrt{-1}}{2}}{0}{\frac{\sqrt{-1}}{2}}{-\frac{1}{2}}{0}{0}{0}{0}
\threematrix{0}{-1}{0}{-1}{0}{0}{0}{0}{-1}=
\threematrix{-\frac{1}{2}}{\frac{\sqrt{-1}}{2}}{0}{-\frac{\sqrt{-1}}{2}}{-\frac{1}{2}}{0}{0}{0}{0}, $$
the above two matrices are congruent. So, up to orthogonal congruences, we have
$$A=\threematrix{-\frac{1}{2}}{\frac{\sqrt{-1}}{2}}{0}{-\frac{\sqrt{-1}}{2}}{-\frac{1}{2}}{0}{0}{0}{0}.$$
It is straightforward to check that it gives a solution of Eq.~(\ref{eq:mateq}).
\smallskip

\noindent
{\em Case 3.4: }$ \frac{1}{2}(A+A') =\threematrix{\lambda}{0}{0}{0}{\sqrt{-1}}{1}{0}{1}{-\sqrt{-1}}, \lambda\neq0.$

In this case we have $a_{11}=\lambda\neq 0$ and $a_{22}+a_{33}=\sqrt{-1}-\sqrt{-1}=0$. By a similar argument as in Case 3.3, we see that Eq.~(\ref{eq:33a}) holds.
Thus $a_{22}+a_{33}=(-k^2+k^2+1)a_{11}$, implying $a_{11}=0$. This is a contradiction. So this case does not give any solution of Eq.~(\ref{eq:mateq}).

\smallskip

\noindent
{\em Case 3.5: } $ \frac{1}{2}(A+A') =\threematrix{\lambda+\sqrt{-1}}{1}{0}{1}{\lambda-\sqrt{-1}}{0}{0}{0}{0}, \lambda\neq0.$
\smallskip

\noindent

Since $\mathrm{tr}(\frac{1}{2}(A+A'))=-1\smallskip$, we have $\lambda=-\frac{1}{2}.$ So $a_{11}=-\frac{1}{2}+\sqrt{-1}\neq0$. Assume $a_{21}=k a_{11}$.
Since $a_{13}=-a_{31}$ and $A'A=0$, we obtain
$$A=\threematrix{a_{11}}{a_{12}}{\mp\sqrt{-(k^2+1)}a_{11}}{k a_{11}}{k a_{12}}{\mp k\sqrt{-(k^2+1)}a_{11}}{\pm\sqrt{-(k^2+1)}a_{11}}{\pm\sqrt{-(k^2+1)}a_{12}}{(k^2+1)a_{11}}.$$
Moreover, by assumption, we have
$$a_{33}=0,\quad a_{11}=-\frac{1}{2}+\sqrt{-1},\quad a_{22}=-\frac{1}{2}-\sqrt{-1},\quad a_{12}+a_{21}=2.$$ Therefore
$$A=\threematrix{-\frac{1}{2}+\sqrt{-1}}{1-\frac{\sqrt{-1}}{2}}{0}{1+\frac{\sqrt{-1}}{2}}{-\frac{1}{2}-\sqrt{-1}}{0}{0}{0}{0}.$$
It is straightforward to check that it gives a solution of Eq.~(\ref{eq:mateq}).
\smallskip

\noindent

{\em Case 3.6: } $\frac{1}{2}(A+A')=\threematrix{0}{1+\sqrt{-1}}{0}{1+\sqrt{-1}}{0}{1-\sqrt{-1}}{0}{1-\sqrt{-1}}{0}$

In this case, $\mathrm{tr}(\frac{1}{2}(A+A'))=0$ which is a contradiction to our assumption that $\mathrm{tr}(\frac{1}{2}(A+A'))=\mathrm{tr}(A)=-1$. So this case does not give any solution of Eq.~(\ref{eq:mateq}).
\medskip

We have now completed the proof of Theorem~\ref{thm:matclass}.

\section{Appendix: Proof of Lemma~\ref{le:ag}}

Let $\varphi:\mathrm{sl}(2,\complex)\rightarrow \mathrm{sl}(2,\complex)$ be an automorphism satisfying Eq.~(\ref{eq:mor1}). Suppose $T$ is the matrix of $\varphi$ with respect to the basis $\{e_1,e_2,e_3\}$ as in Eq.~(\ref{eq:base}), that is,
$$\varphi(e_i)=\sum_{j=1}^3t_{ij}e_j,\;\;{\rm where}\;\;T=(t_{ij}),\;\;
  t_{ij}\in \complex,i,j=1,2,3.$$
Substituting the above equation into Eq.~(\ref{eq:mor1}), we obtain
\begin{eqnarray*}
\threevec{{[\varphi (e_2),\varphi (e_3)]}}{{[\varphi (e_3),\varphi (e_1)]}}{{[\varphi (e_1),\varphi (e_2)]}}
&=&\threevec{{[t_{21}e_1+t_{22}e_2+t_{23}e_3,\quad t_{31}e_1+t_{32}e_2+t_{33}e_3]}}
{{[t_{31}e_1+t_{32}e_2+t_{33}e_3,\quad t_{11}e_1+t_{12}e_2+t_{13}e_3]}}
{{[t_{11}e_1+t_{12}e_2+t_{13}e_3,\quad t_{21}e_1+t_{22}e_2+t_{23}e_3]}}\\
&=&\threevec
{\twodet{t_{22}}{t_{23}}{t_{32}}{t_{33}}e_1 +\twodet{t_{23}}{t_{21}}{t_{33}}{t_{31}}e_2 +\twodet{t_{21}}{t_{22}}{t_{31}}{t_{32}}e_3
\medskip}
{\twodet{t_{32}}{t_{33}}{t_{12}}{t_{13}}e_1 +\twodet{t_{33}}{t_{31}}{t_{13}}{t_{11}}e_2 +\twodet{t_{31}}{t_{32}}{t_{11}}{t_{12}}e_3 \medskip}
{\twodet{t_{12}}{t_{13}}{t_{22}}{t_{23}}e_1 +\twodet{t_{13}}{t_{11}}{t_{23}}{t_{21}}e_2 +\twodet{t_{11}}{t_{12}}{t_{21}}{t_{22}}e_3}
=(T^*)' \threevec{e_1}{e_2}{e_3}.
\end{eqnarray*}
On the other hand, by Eq.~(\ref{eq:eprod}), we have
$$\threevec{\varphi([e_2,e_3])}{\varphi([e_3,e_1])}{\varphi([e_1,e_2])}=\threevec{\varphi(e_1)}{\varphi(e_2)}{\varphi(e_3)}
=T \threevec{e_1}{e_2}{e_3}.$$
Thus
$$(T^*)'=T.$$
So $T^*=T'$. Thus $$\det T=\det T'=\det T^*=(\det T)^2.$$
Then $\det T=1$, and $$T'=T^*=\det T T^{-1}=T^{-1}.$$
That is, $T T'=T' T=I_3$ and $\det T=1$. So $T\in \mathrm{SO}(3,\complex)$.

Conversely,
let $T=(t_{ij})\in \mathrm{SO}(3,\complex)$.  Define a linear operator $\psi$ on $\mathrm{sl}(2,\complex)$ by
$$\psi(e_i)=\sum_{j=1}^3 t_{ij}e_j.$$
Since $T$ is invertible, $\psi$ is a bijection.
Similar to the calculating above, there is
\begin{eqnarray*}
\threevec{{[\psi(e_2),\psi(e_3)]}}{{[\psi(e_3),\psi(e_1)]}}{{[\psi(e_1),\psi(e_2)]}}
&=&\threevec{{[t_{21}e_1+t_{22}e_2+t_{23}e_3,\quad t_{31}e_1+t_{32}e_2+t_{33}e_3]}}
{{[t_{31}e_1+t_{32}e_2+t_{33}e_3,\quad t_{11}e_1+t_{12}e_2+t_{13}e_3]}}
{{[t_{11}e_1+t_{12}e_2+t_{13}e_3,\quad t_{21}e_1+t_{22}e_2+t_{23}e_3]}}\\
&=&(T^*)'\threevec{e_1}{e_2}{e_3}=T \threevec{e_1}{e_2}{e_3} =\threevec{\psi([e_2,e_3])}{\psi([e_3,e_1])}{\psi([e_1,e_2])}
\end{eqnarray*}
Since $\{e_1,e_2,e_3\}$ is a basis of $\mathrm{sl}(2,\complex)$, we have $[\psi(x),\psi(y)]=\psi([x,y])$ for all $x,y\in \mathrm{sl}(2,\complex)$.
Thus $\psi$ is an automorphism of the Lie algebra $(\mathrm{sl}(2,\complex),[,])$.

\bigskip
{\bf Acknowledgement.}
Chengming Bai thanks the support by NSFC (10920161), NKBRPC (2006CB805905) and SRFDP
(200800550015). Li Guo thanks NSF grant DMS 1001855
for support. Yu Pan and Qing Liu thank Dr. Fuhai Zhu for his valuable discussions.

\end{document}